\documentclass[12pt]{amsart}
\usepackage[margin=1.2in]{geometry}
\usepackage{graphicx,latexsym}
\usepackage{comment}
\usepackage{tikz}
\usepackage{amsfonts, amssymb, amsmath, amsthm, bm, hyperref}
\usepackage{tikz}
\usetikzlibrary{matrix,arrows}
\allowdisplaybreaks

\newcommand{\N}{\mathbb{N}}
\newcommand{\Z}{\mathbb{Z}}
\newcommand{\R}{\mathbb{R}}
\newcommand{\C}{\mathbb{C}}

\newcommand{\dt}{{\rm d}t }

\newcommand{\beq}{\begin{eqnarray}}
\newcommand{\eeq}{\end{eqnarray}}

\newcommand{\beqs}{\begin{eqnarray*}}
\newcommand{\eeqs}{\end{eqnarray*}}

\newtheorem{theorem}{Theorem}[section]
\newtheorem{proposition}[theorem]{Proposition}
\newtheorem{lemma}[theorem]{Lemma}
\newtheorem{corollary}[theorem]{Corollary}

\theoremstyle{definition}

\newtheorem{example}[theorem]{Example}

\newtheorem{remark}[theorem]{Remark}

\numberwithin{equation}{section}

\begin{document}

\title[Sequence space representations of Beurling-Bj\"orck spaces]{Sequence space representations of Beurling-Bj\"orck spaces via Gabor frames and Wilson bases}

\author[A. Debrouwere]{Andreas Debrouwere}
\address{Department of Mathematics and Data Science \\ Vrije Universiteit Brussel, Belgium\\ Pleinlaan 2 \\ 1050 Brussels \\ Belgium}
\email{andreas.debrouwere@vub.be}

\author[L. Neyt]{Lenny Neyt}
\address{University of Vienna\\ Faculty of Mathematics\\ Oskar-Morgenstern-Platz 1 \\ 1090 Wien\\ Austria}
\thanks{The research of L. Neyt was funded in whole by the Austrian Science Fund (FWF) 10.55776/ESP8128624. For open access purposes, the author has applied a CC BY public copyright license to any author-accepted manuscript version arising from this submission.}
\email{lenny.neyt@univie.ac.at}


\begin{abstract}
We establish sequence space representations of a broad class of Beurling-Bj\"orck spaces $\mathcal{S}^{(\omega)}_{(\eta)}$ and $\mathcal{S}^{\{\omega\}}_{\{\eta\}}$.  We develop two different approaches: a non-constructive one based on Gabor frames and the structure theory of Fréchet spaces, and a constructive one using Wilson bases, under stronger assumptions on the defining weight functions $\omega$ and $\eta$. As an application, we provide an isomorphic classification of the spaces  $\mathcal{S}^{(\omega)}_{(\eta)}$ and $\mathcal{S}^{\{\omega\}}_{\{\eta\}}$ in terms of $\omega$ and $\eta$. In particular, our results are applicable to the classical Gelfand-Shilov spaces $\mathcal{S}^\mu_\tau$ for $\mu, \tau \geq 1/2$ (non-constructive approach) and $\mu, \tau \geq 1$ (constructive approach).
\end{abstract}

\maketitle

\section{Introduction}
In his classical work \cite{S-ThDist}, Schwartz characterized the space $\mathcal{S}$ of rapidly decreasing smooth functions on $\mathbb{R}$ via the Hermite series expansions of its elements. As a consequence, he showed that $\mathcal{S}$ is topologically isomorphic to the space $s$ of rapidly decreasing sequences. Later, in the early 1980s, Valdivia and Vogt established sequence space representations -- that is, topological isomorphisms with sequence spaces -- of most of the function and distribution spaces arising in Schwartz's theory of distributions \cite{V-TopLCS, V-SeqSpRepTestFuncDist}. We refer to \cite{Bargetz2015a, Bargetz2015b, B-D-N-SeqSpRepWilson, D-N-SeqSpRepTMIB, D-N-V-ExplCommSeqSpRep, OrtnerWagner2013} for some recent works on this topic. Beyond their intrinsic interest, sequence space representations of (generalized) function spaces play a key role in the study of Schauder bases and in the isomorphic classification of such spaces.

In this article, we obtain sequence space representations of Beurling–Björck spaces defined on $\R$ \cite{B-LinPDOGenDist,D-N-V-CharNuclBBSp}. These are function spaces whose elements and their Fourier transforms satisfy decay estimates governed by (possibly different) weight functions. Beurling–Björck spaces encompass a wide variety of spaces of ultradifferentiable functions satisfying global decay estimates \cite{C-C-K-CharGelfandShilovFourier}, including the spaces $\mathcal{S}^{\mu}_\tau$ introduced by Gelfand and Shilov \cite{G-S-GenFunc2} and their projective counterparts $\Sigma^{\mu}_\tau$ \cite{C-G-P-R-AnistropicShubinOpExpGSSp, D-SeqSpRepEntFunc}; see Example \ref{CGS} for the precise definition of these spaces. 
The space  $\mathcal{S}^{\mu}_\tau$ is non-trivial if and only $\mu + \tau \geq 1$, while $\Sigma^{\mu}_\tau$ is non-trivial if and only $\mu + \tau > 1$ (cf.\ \cite[Section 8]{G-S-GenFunc2}).

Sequence space representations of Fourier invariant Gelfand–Shilov type spaces, including  $\mathcal{S}^{\mu}_\mu$ ($\mu \geq 1/2$) and $\Sigma^{\mu}_\mu$ ($\mu > 1/2$), were obtained by Langenbruch \cite{L-HermiteFuncWeighSpGenFunc} via Hermite expansions. For non-Fourier invariant spaces, establishing such representations appears to be more difficult. Sequence space representations of $\mathcal{S}^{\mu}_\tau$ ($\mu + \tau  \geq 1$)  and $\Sigma^{\mu}_\tau$ ($\mu + \tau  > 1$) with $\mu/\tau \in \mathbb{Q}$ follow from the work of Cappiello et al.\ \cite{C-G-P-R-AnistropicShubinOpExpGSSp}, where these spaces are characterized via eigenfunction expansions with respect to certain anisotropic Shubin operators. Langenbruch \cite{L-BasesGerms, L-DiamDimWeighSpGerms} obtained sequence space representations of a class of weighted spaces of holomorphic germs near $\mathbb{R}$, including the spaces $\mathcal{S}^{1}_\tau$ ($\tau >0$). His method combines results from the structure theory of Fréchet spaces with techniques from complex analysis. The first author \cite{C-G-P-R-AnistropicShubinOpExpGSSp} applied a similar approach to find sequence space representations of certain weighted spaces entire function spaces, including $\Sigma^{1}_\tau$ ($\tau >0$). A drawback of this method -- as opposed to the one based on eigenfunction expansions -- is that the resulting isomorphisms are non-constructive, as they rely on structural results for power series spaces which are themselves non-constructive.

We provide sequence space representations of a broad class of Beurling-Bj\"orck spaces that are not necessarily Fourier invariant, using tools from time-frequency analysis \cite{G-FoundationsTimeFreqAnal}, namely, Gabor frames and Wilson bases.   Gabor frames, a class of frames for $L^2(\R)$ consisting of time-frequency shifts of a single window function, are fundamental in time-frequency analysis and are particularly well-suited for describing weighted spaces of (ultra)differentiable functions in a discrete fashion \cite{GZ-SpTestFuncSTFT}. 
Due to the Balian-Low `no-go' theorem \cite[Section 8.4]{G-FoundationsTimeFreqAnal}, Gabor frames generated by windows sufficiently well-localized in both time and frequency cannot form Riesz bases. Remarkably, Daubechies, Jaffard, and Journ\'e \cite{D-J-J-WilsonBasisExpDecay} constructed orthonormal bases for $L^2(\R)$, called Wilson bases, consisting of linear combinations of time-frequency shifts of a single window.

We establish sequence space representations via two different methods. Firstly, following our previous work \cite{D-N-SeqSpRepTMIB}, we obtain such representations by combining the Pe\l czynski-Vogt decomposition method \cite{V-IsomorphPowSp} -- an abstract result about power series spaces -- with structural properties of Gabor frames and characterizations of Beurling-Bj\"orck spaces via Gabor frame coefficients. Due to the use of the Pe\l czynski-Vogt decomposition method, this yields non-constructive representations. Secondly, following \cite{B-D-N-SeqSpRepWilson}, we obtain constructive and explicit sequence space representations using characterizations of Beurling-Bj\"orck spaces via Wilson basis coefficients. 

In both the non-constructive and constructive approach, we need to impose some growth conditions on the weight functions defining the Beurling–Björck spaces. These conditions are necessary to apply certain results about Gabor frames \cite{Janssen} and Wilson bases \cite{D-J-J-WilsonBasisExpDecay}. The assumptions are less restrictive in the non-constructive case, which motivates our consideration of both approaches.

As a sample, we now state our main results for the spaces  $\mathcal{S}^\mu_\tau$ and $\Sigma^\mu_\tau$. We need to introduce some notation about power series spaces  \cite[Chapter 29]{M-V-IntroFuncAnal}; see also Section \ref{sec:PowerSeries}. 
Let $\alpha = (\alpha_n)_{n \in \N}$  and $\beta = (\beta_n)_{n \in \N}$ be non-negative sequences tending to infinity. We define $\Lambda_\infty(\alpha)$ ($\Lambda_0(\alpha)$) as the Fr\'echet space consisting of all $(c_n)_{n \in \N} \in \C^\N$ such that
$$
\sup_{n \in \N}|c_n| e^{r\alpha_n} < \infty
$$
for all $r >0$ (for all $r <0$). Similarly, we define $\Lambda_\infty(\alpha,\beta)$ ($\Lambda_0(\alpha,\beta)$) as the Fr\'echet space consisting of all $(c_{k,n})_{(k,n) \in \N^2} \in \C^{\N^2}$ such that
$$
\sup_{(k,n) \in \N^2}|c_{k,n}| e^{r(\alpha_k + \beta_n)} < \infty
$$
for all $r >0$ (for all $r <0$). Given two locally convex spaces $E$ and $F$, we write $E \cong F$ to indicate that they are topologically isomorphic. Our first main result is:

\begin{theorem} \label{thm-intro}
 For  $\mu, \tau >1/2$, it holds that
\begin{equation}
\label{projcase}
\Sigma^{\mu}_{\tau} \cong \Lambda_\infty(k^{\frac{1}{\mu}},n^{\frac{1}{\tau}} ) \cong \Lambda_\infty(n^{\frac{1}{\mu + \tau}}).
\end{equation}
For  $\mu, \tau \geq 1/2$, it  holds that
\begin{equation}
\label{LBcase}
\mathcal{S}^{\mu}_{\tau} \cong \Lambda'_0(k^{\frac{1}{\mu}}, n^{\frac{1}{\tau}}) \cong \Lambda'_0(n^{\frac{1}{\mu + \tau}}).
\end{equation}
\end{theorem}

Theorem \ref{thm-intro} will be shown by using the non-constructive approach outlined above. 
We will also show the first isomorphisms in \eqref{projcase} and \eqref{LBcase} in a constructive and explicit way, under the stronger assumptions    $\mu, \tau >1$ and   $\mu, \tau \geq 1$, respectively. As a consequence, we obtain  `commuting' sequence space representations of such spaces:

\begin{theorem}\label{intro-2}
There exists a topological isomorphism
  $$
   \Psi\colon   \mathcal{S}  \to  \Lambda_\infty(\log(1+k), \log(1+n))
$$
such that, for all $\mu, \tau > 1$, 
$$ 
   \Psi_{\mid \Sigma^{\mu}_{\tau}} \colon   \Sigma^{\mu}_{\tau}   \to  \Lambda_\infty(k^{\frac{1}{\mu}},n^{\frac{1}{\tau}} )
$$
and, for all $\mu, \tau \geq 1$, 
$$ 
   \Psi_{ \mid \mathcal{S}^{\mu}_{\tau}} \colon   \mathcal{S}^{\mu}_{\tau}   \to \Lambda'_0(k^{\frac{1}{\mu}},n^{\frac{1}{\tau}} )
$$
are topological isomorphisms.
\end{theorem}

This article is organized as follows. In the preliminary Sections \ref{sec:Spaces} and  \ref{sec:PowerSeries}, we introduce Beurling-Bj\"orck spaces and power series spaces and recall some of their fundamental properties.  Gabor frames and Wilson bases are discussed in Section \ref{sec:TFAnalysis}, where we also study their mapping properties on Beurling-Bj\"orck spaces.
The main results of this paper, namely, non-constructive and constructive sequence space representations of Beurling-Bj\"orck spaces, are established in Section \ref{sec:SeqSpRep}.
In the final Section \ref{sec:IsoClass}, we apply these representations to study the isomorphic classification of Beurling-Bj\"orck spaces.

To end our introduction, we mention that in this work, we only consider the one-dimensional case to avoid some tedious technicalities. However, by the use of tensor product decompositions, which are known for the Beurling-Bj\"orck spaces we consider \cite[Theorem 1.1]{N-V-KernThmBBSp}, it is possible to extend our results to higher dimensions. See also \cite[Section 12.3]{G-FoundationsTimeFreqAnal}, where a similar approach has been taken.

\section{Spaces of ultradifferentiable functions with rapid decay}
\label{sec:Spaces}

Let $\omega: [0,\infty) \rightarrow [0,\infty)$ be continuous and increasing. We consider the following conditions on $\omega$:
\begin{itemize}
\item[$(\alpha)$] $\omega(2t) = O(\omega(t))$.
\item [$(\gamma)$] $\log t = O(\omega(t))$.
\item [$\{\gamma\}$] $\log t = o(\omega(t))$.
\end{itemize}
Condition $(\alpha)$ is equivalent to the existence of $K,C \geq 1$ such that
\begin{equation}
\label{eq:AlphaEquiv}
\omega(t_1 + t_2) \leq K[\omega(t_1) + \omega(t_2)]  + \log C, \qquad \forall t_1, t_2 \geq 0.
\end{equation}
We call $\omega$ \emph{non-quasianalytic} if 
$$
\int_0^\infty \frac{\omega(t)}{1+t^2} \dt < \infty.
$$
\begin{example}
$(i)$ Let $\mu >0$ and $u \in \R$. Then, $\omega_{\mu, u}(t) = t^{\frac{1}{\mu}}\log(1+ t)^u$ satisfies $(\alpha)$ and $\{\gamma\}$. The function $\omega_{\mu, u}$ is non-quasinalytic if and only if $\mu  > 1$ or $\mu =1$ and $u < -1$.  \\
$(ii)$ Let $u \geq 1$. Then, $\omega_u(t) = \log(1+ t)^u$ satisfies $(\alpha)$ and $(\gamma)$, and is non-quasianalytic. The function $\omega_u$ satisfies $\{\gamma\}$ if and only if $u > 1$.
\end{example}

Let $\omega, \eta: [0,\infty) \rightarrow [0,\infty)$ be continuous and increasing. For $h > 0$  and $\varphi: \R \to \C$ we set
$$
\|\varphi\|_{\omega,h} = \sup_{x\in\mathbb{R}} |\varphi(x)|e^{h \omega(x)}.
$$
 We define  $\mathcal{S}_{\eta,h}^{\omega,h}$ as the normed space consisting of all $\varphi \in L^1(\mathbb{R})$ such that 
$$
\|\varphi\|_{\mathcal{S}_{\eta,h}^{\omega,h}} = \max \{ \|\varphi\|_{\eta,h}, \|\widehat{\varphi}\|_{\omega,h}\}<\infty.
$$
We define the \emph{Beurling-Bj\"{o}rck spaces (of Beurling and Roumieu type)} as 
\[
\mathcal{S}_{(\eta)}^{(\omega)}=\varprojlim_{h \to\infty} \mathcal{S}_{\eta,h}^{\omega,h} \qquad \mbox{and}
\qquad \mathcal{S}_{\{\eta\}}^{\{\omega\}}=\varinjlim_{h\to0^{+}} \mathcal{S}_{\eta,h}^{\omega,h}.
\]
We shall use $\mathcal{S}_{[\eta]}^{[\omega]}$ as a common notation for $\mathcal{S}_{(\eta)}^{(\omega)}$ and  $\mathcal{S}_{\{\eta\}}^{\{\omega\}}$. A similar convention will be used for other notations. In addition, we shall sometimes first state assertions for the Beurling case, followed in parentheses by the corresponding ones for the Roumieu case. 

 A continuous increasing function $\omega: [0,\infty) \rightarrow [0,\infty)$  is called a \emph{Beurling (Roumieu) weight function}  if  $\omega$ satisfies $(\alpha)$ and $(\gamma)$ ($(\alpha)$ and $\{\gamma\}$). If we simply write that $\omega$ is a weight function, we mean that $\omega$ is a Beurling  (Roumieu) weight function when we consider the Beurling (Roumieu) case.

Let $\omega$ be a weight function. We define $\mathcal{D}^{[\omega]}$ as the space consisting of all $\varphi \in L^1(\R)$ with compact support such that $\|\widehat{\varphi}\|_{\omega,h} < \infty$ for all $h > 0$ (some $h > 0$). The space $\mathcal{D}^{[\omega]}$ is non-trivial if and only if $\omega$ is non-quasianalytic \cite[Theorem 1.3.7]{B-LinPDOGenDist}. 

Let $\omega_i$ and $\eta_i$, $i = 1,2$, be four weight functions. Then, $\mathcal{S}_{[\eta_1]}^{[\omega_1]} \subseteq \mathcal{S}_{[\eta_2]}^{[\omega_2]}$ if $\omega_2 = O(\omega_1)$ and  $\eta_2 = O(\eta_1)$, and $\mathcal{S}_{\{\eta_1\}}^{\{\omega_1\}} \subseteq \mathcal{S}_{(\eta_2)}^{(\omega_2)}$  if $\omega_2 = o(\omega_1)$ and  $\eta_1 = o(\eta_2)$. Both inclusions are continuous. In particular, $\mathcal{S}_{[\eta_1]}^{[\omega_1]} = \mathcal{S}_{[\eta_2]}^{[\omega_2]}$ as locally convex spaces if $\omega_1 \asymp \omega_2$  (meaning that $\omega_1= O(\omega_2)$ and  $\omega_2 = O(\omega_1)$) and  $\eta_1  \asymp \eta_2$.

Let $\omega$ and $\eta$ be two weight functions. Since $\mathcal{S}^{(\log(1+ t))}_{(\log(1+ t))} = \mathcal{S}$, it holds that $\mathcal{S}_{[\eta]}^{[\omega]} \subseteq \mathcal{S}$ continuously. Using this, one can show that  $\mathcal{S}_{(\eta)}^{(\omega)}$ is a Fr\'echet space and that $\mathcal{S}_{\{\eta\}}^{\{\omega\}}$ is an $(LB)$-space. 
The Fourier transform is a topological isomorphism from $\mathcal{S}_{[\eta]}^{[\omega]}$ onto $\mathcal{S}^{[\eta]}_{[\omega]}$.
 
 \begin{example}\label{CGS}
Let $\mu, \tau >0$. We now define the classical Gelfand-Shiov spaces $\Sigma^{\mu}_{\tau}$ and $\mathcal{S}^{\mu}_{\tau}$, already considered in the introduction. For $k >0$ we write $\mathcal{S}^{\mu,k}_{\tau,k}$  as the Banach space consisting of all $\varphi \in C^\infty(\mathbb{R})$ such that
\begin{equation}
\label{GScl}
\| \varphi\|_{\mathcal{S}^{\mu,k}_{\tau,k}} = \sup_{p,q \in \N} \sup_{x \in \R} \frac{|\varphi^{(p)}(x) x^q|}{k^{p+q}p!^\mu q!^\tau} < \infty.
\end{equation}
We set
$$
\Sigma^{\mu}_{\tau}=\varprojlim_{k \to 0^+} \mathcal{S}^{\mu,k}_{\tau,k} \qquad \mbox{and}
\qquad \mathcal{S}^{\mu}_{\tau}=\varinjlim_{k \to \infty} \mathcal{S}^{\mu,k}_{\tau,k} .
$$
Then, $\Sigma^\mu_\tau = \mathcal{S}^{(t^{1/\mu})}_{(t^{1/\tau})}$ and $\mathcal{S}^\mu_\tau = \mathcal{S}^{\{t^{1/\mu}\}}_{\{t^{1/\tau}\}}$ as locally convex spaces \cite[Corollary 2.5]{C-C-K-CharGelfandShilovFourier}\footnote{In \cite[Corollary 2.5]{C-C-K-CharGelfandShilovFourier} only the Roumieu case is considered, but the Beurling case can be treated similarly.}.
 \end{example}

\begin{remark}
Let $\omega$ and $\eta$ be two weight functions. Characterizing the non-triviality of the space $\mathcal{S}_{[\eta]}^{[\omega]}$ in terms of $\omega$ and $\eta$ appears to be an open problem. It is clear that $\mathcal{S}_{[\eta]}^{[\omega]}$ is non-trivial if either $\omega$ or $\eta$ is non-quasianalytic.  As already mentioned in the introduction, $\Sigma^{\mu}_{\tau}$ is non-trivial if and only if $ \mu + \tau >1$, while $\mathcal{S}^{\mu}_{\tau}$ is non-trivial if and only $\mu + \tau \geq 1$ (cf.\ \cite[Section 8]{G-S-GenFunc2}).
\end{remark}

\section{Power series spaces}
\label{sec:PowerSeries}
Let $I$ be a countable index set and let $\alpha = (\alpha_i)_{i \in I}$ be a sequence of non-negative numbers. 
Given $r \in \R$, we write $\Lambda_r^\alpha(I)$ for the Banach space consisting of all  $c = (c_i)_{i \in I} \in \C^I$ such that 
$$
\| c \|_{\Lambda_r^\alpha(I)} = 
\sup_{i \in I} |c_i| e^{r\alpha_i}  < \infty.
$$
We define the Fr\'echet spaces 
$$
\Lambda_\infty(I;\alpha) = \varprojlim_{r \to \infty} \Lambda_r^\alpha(I), \qquad \Lambda_0(I;\alpha) = \varprojlim_{r \to 0^-} \Lambda_r^\alpha(I).
$$
 
By an \emph{exponent sequence}, we mean  a non-decreasing sequence $\alpha = (\alpha_n)_{n \in \N}$ of non-negative numbers such that $\alpha_n \rightarrow \infty$ as $n \to \infty$. We define  the \emph{power series (of infinite and finite type)} \cite[Chapter 29]{M-V-IntroFuncAnal} as
$$
\Lambda_\infty(\alpha) =  \Lambda_\infty(\N;\alpha), \qquad \Lambda_0(\alpha) =  \Lambda_0(\N;\alpha).
$$
The space $\Lambda_\infty(\alpha)$ is nuclear if and only if $\log n= O(\alpha_n)$, while $\Lambda_0(\alpha)$ is nuclear if and only if $\log n = o(\alpha_n)$ \cite[Proposition 29.6]{M-V-IntroFuncAnal}. 
The  sequence $\alpha$ is said to be \emph{stable} if $\alpha_{2n} = O(\alpha_{n})$. 
We call $\Lambda_\theta(\alpha)$, $\theta \in \{0, \infty\}$, stable if $\alpha$ is so. 

Let $\alpha$ and $\beta$ be exponent sequences and let $\theta \in \{0, \infty\}$. 
Then, $\Lambda_\theta(\alpha) \subseteq \Lambda_\theta(\beta)$ continuously if  $\beta= O(\alpha)$. 
In particular, $\Lambda_\theta(\alpha) = \Lambda_\theta(\beta)$ as Fr\'echet spaces if $\alpha \asymp \beta$.

We will make use of the following Pe\l czynski-Vogt decomposition type result.

\begin{proposition}[{\cite[Satz 1.4]{V-IsomorphPowSp}}] \label{iso-sequence}
Let $\Lambda_\theta(\alpha)$, $\theta  \in \{0, \infty\}$, be a nuclear stable power series space. If a Fr\'echet space $E$ is isomorphic to a complemented subspace of  $\Lambda_\theta(\alpha)$  and  $\Lambda_\theta(\alpha)$ is isomorphic to a complemented subspace of $E$, then $E \cong  \Lambda_\theta(\alpha)$.
\end{proposition}

Let $\alpha$ and $\beta$ be   exponent sequences and let $\theta  \in \{0, \infty\}$. We set
$$
\Lambda_\theta(\alpha,\beta) = \Lambda_\theta(\N^2;  (\alpha_k+ \beta_n)_{(k,n) \in \N^2}).
$$
 We define the sequence $\alpha \sharp \beta$  as the non-decreasing rearrangement of the set $\{\alpha_k+ \beta_n \, | \, (k,n) \in \N^2\}$. Then,
\begin{equation}
\Lambda_\theta(\alpha,\beta) \cong \Lambda_\theta(\alpha \sharp \beta).
\label{basic}
\end{equation}

We now study some properties of the  $\sharp$-product. Given an exponent sequence $\alpha$, we denote by $\nu_\alpha$ the counting function of $\alpha$, i.e.,
$$
\nu_\alpha(s) = \sum_{\alpha_n \leq s} 1, \qquad s \geq 0.
$$

\begin{lemma} \label{biglittleoh}
Let $\alpha$ and $\beta$ be exponent sequences. Then, $\alpha = O(\beta)$ ($\alpha = o(\beta)$) if and only if for some $L > 0$ (for all $L > 0$) 
$$
\nu_\beta(s) \leq \nu_\alpha(Ls), \qquad \forall s \geq 0.
$$
\end{lemma}
\begin{proof}
Clear.
\end{proof}

\begin{lemma}\label{stablec} 
An exponent sequence $\alpha$ is stable if and only if there is $L > 0$ such that 
$$
2\nu_\alpha(s) \leq \nu_\alpha(Ls), \qquad \forall s \geq 0.
$$
\end{lemma}
\begin{proof}
Clear.
\end{proof}

\begin{lemma} \label{fundamental}
Let $\alpha$ and $\beta$ be exponent sequences. Then,
$$
\nu_\alpha(s/2) \nu_\beta(s/2) \leq \nu_{\alpha \sharp \beta}(s) \leq \nu_\alpha(s) \nu_\beta(s),  \qquad  \forall s \geq 0.
$$
\end{lemma}
\begin{proof}
 Let $s \geq 0$ be arbitrary. Note that
$$
\nu_{\alpha \sharp \beta}(s) = \sum_{\alpha_k + \beta_n \leq s} 1 = \sum_{\beta_n \leq s} \sum_{\alpha_k \leq s- \beta_n} 1 =  \sum_{\beta_n \leq s} \nu_\alpha(s- \beta_n).
$$
Hence,
$$
\nu_{\alpha \sharp \beta}(s) \leq  \nu_\alpha(s) \sum_{\beta_n \leq s} 1 =  \nu_\alpha(s) \nu_\beta(s)
$$
and
$$
\nu_{\alpha \sharp \beta}(s) \geq  \sum_{\beta_n \leq s/2} \nu_\alpha(s- \beta_n) \geq  \nu_\alpha(s/2) \sum_{\beta_n \leq s/2} 1 = \nu_\alpha(s/2) \nu_\beta(s/2).
$$
\end{proof}

\begin{lemma}\label{properties-sharp}
Let $\alpha$ and $\beta$ be exponent sequences and let $\theta \in \{0, \infty\}$.
\begin{itemize}
\item[$(i)$] If both $\Lambda_\theta(\alpha)$ and  $\Lambda_\theta(\beta)$ are nuclear, then $\Lambda_\theta(\alpha \sharp \beta)$ is also nuclear.
\item[$(ii)$] If either $\Lambda_\theta(\alpha)$ or $\Lambda_\theta(\beta)$ is stable, then $\Lambda_\theta(\alpha \sharp \beta)$ is also stable.
\end{itemize}
\end{lemma}
\begin{proof}
$(i)$ Recall that, given an exponent sequence $\gamma$, $\Lambda_\infty (\gamma)$ ($\Lambda_0(\gamma)$) is nuclear if and only if  $\log n= O(\gamma_n)$ ($\log n = o(\gamma_n)$). Hence, the result follows from Lemmas \ref{biglittleoh} and \ref{fundamental}. \\
$(ii)$ This follows from Lemmas \ref{stablec} and \ref{fundamental}. \\
\end{proof}

Let $\omega$ be a weight function. We define $ \alpha(\omega)= (\omega(n))_{n \in \N}$. Since $\omega$ satisfies $(\alpha)$ and $(\gamma)$ $(\{\gamma\})$,  $\alpha(\omega)$ is a stable exponent sequence that satisfies
 $\log n= O(\alpha(\omega)_n)$ ($\log n = o(\alpha(\omega)_n)$). 

Given two weight functions $\omega$ and $\eta$, we define $\alpha(\omega,\eta) =((\omega^{-1}\eta^{-1})^{-1}(n))_{n \in \N}$. Due to condition $(\alpha)$, $\alpha(\omega,\eta)$ is a stable exponent sequence.

 \begin{lemma}\label{explicit}
Let $\omega$ and $\eta$ be two weight functions. Then,  
$$
\alpha(\omega) \sharp \alpha(\eta) \asymp \alpha(\omega,\eta).
$$ 
 \end{lemma}
 \begin{proof}
Since $\alpha(\omega,\eta)$  is stable, Lemma \ref{stablec} implies that there are $L > 0$ and $s_0 \geq 0$ such that
$$
\nu_{\alpha(\omega,\eta)}(s) + 1\leq 2\nu_{\alpha(\omega,\eta)}(s)  \leq \nu_{\alpha(\omega,\eta)}(Ls), \qquad \forall s \geq s_0.
$$
By Lemma \ref{fundamental}, we obtain that, for all $s \geq s_0$,
$$
\nu_{\alpha(\omega) \sharp \alpha(\eta)}(s) \leq \nu_{\alpha(\omega)}(s) \nu_{\alpha(\eta)}(s)  \leq \omega^{-1}(s)\eta^{-1}(s) \leq \nu_{\alpha(\omega,\eta)}(s) + 1\leq \nu_{\alpha(\omega,\eta)}(Ls).
$$
Hence, $\alpha(\omega,\eta) = O(\alpha(\omega) \sharp \alpha(\eta))$ by Lemma \ref{biglittleoh}. 

Since $\alpha(\omega)$ and $\alpha(\eta)$  are stable, Lemma \ref{stablec} implies that there are $L > 0$ and $s_0 \geq 0$ such that 
\begin{align*}
&&\nu_{\alpha(\omega)}(s) + 1\leq 2\nu_{\alpha(\omega)}(s)  \leq \nu_{\alpha(\omega)}(Ls), \qquad \forall s \geq s_0, \\
&&\nu_{\alpha(\eta)}(s) + 1\leq 2\nu_{\alpha(\eta)}(s)  \leq \nu_{\alpha(\eta)}(Ls), \qquad \forall s \geq s_0.
\end{align*}
By Lemma \ref{fundamental}, we obtain that, for all $s \geq s_0$,
\begin{align*}
 \nu_{\alpha(\omega,\eta)}(s) &\leq \omega^{-1}(s)\eta^{-1}(s) \leq (\nu_{\alpha(\omega)}(s) + 1)(\nu_{\alpha(\eta)}(s) + 1) \\
 & \leq \nu_{\alpha(\omega)}(Ls)\nu_{\alpha(\eta)}(Ls) \leq \nu_{\alpha(\omega) \sharp \alpha(\eta)}(2Ls).
\end{align*}
Hence, $\alpha(\omega) \sharp \alpha(\eta) = O(\alpha(\omega,\eta))$, as follows from Lemma \ref{biglittleoh}. 
\end{proof}

\begin{example}\label{exampleGStwist}
Consider  $\omega_{\mu,u}(t) = t^{\frac{1}{\mu}}\log(1+t)^u$ for $\mu >0$ and $u \in \R$. Then, for $\mu, \tau >0$ and $u, v \in \R$,
$$
 \alpha(\omega_{\mu,u}, \omega_{\tau,v}) \asymp  \alpha(\omega_{\mu +\tau, \frac{u \mu + v \tau}{\mu + \tau}}).
 $$ 
\end{example}

\section{Tools from time-frequency analysis}
\label{sec:TFAnalysis}
In this section, we introduce Gabor frames and Wilson bases;  see the book \cite{G-FoundationsTimeFreqAnal} for more information. The translation and modulation operators are denoted by $T_{x} f(t) = f(t - x)$ and $M_{\xi} f(t) = e^{2 \pi i \xi t} f(t)$, for $x, \xi \in \R$.

\subsection{Gabor frames}	
Let  $\psi \in L^{2}(\R)$ and $a, b > 0$. The set of time-frequency shifts
	\[ \mathcal{G}(\psi, a, b) = \{ T_{ak} M_{bn} \psi \mid (k,n) \in \Z^2 \} \]
is called a \emph{Gabor frame} for $L^{2}(\R)$ if there exist $A, B > 0$ such that
	\[ A \|f\|^{2}_{L^{2}} \leq \sum_{ (k,n) \in \Z^2} \left| (f, T_{ak}M_{bn} \psi)_{L^2}\right|^{2} \leq B \|f \|^{2}_{L^{2}} , \qquad \forall f  \in L^{2}(\R). \]
Let $\psi \in \mathcal{S}$. The \emph{analysis operator}
$$
C_\psi = C^{a,b}_\psi:  L^2(\R) \rightarrow \ell^2(\Z^2), ~  f \mapsto ((f, T_{ak}M_{bn} \psi)_{L^2})_{(k,n) \in \Z^2},
$$
and the \emph{synthesis operator}
$$
D_\psi = D^{a,b}_\psi: \ell^2(\Z^2) \rightarrow L^2(\R), ~  (c_{k,n})_{ (k,n) \in \Z^2} \mapsto \sum_{ (k,n) \in \Z^2} c_{k,n} T_{ak} M_{bn} \psi
$$
are continuous. 
Let $\gamma \in \mathcal{S}$. We call  \emph{$(\psi,\gamma)$ a pair of dual windows  (on $a\Z \times b\Z$)} if
\begin{equation}
D_\gamma \circ C_\psi = \operatorname{id}_{L^2(\R)}.
		\label{comp11}
\end{equation}
In such a case, $(\gamma, \psi)$ is also a pair of dual windows and both $\mathcal{G}(\psi, a, b)$ and $\mathcal{G}(\gamma, a, b)$ are Gabor frames. 

We will use the following characterization of pairs of dual windows,  known as the Wexler-Raz biorthogonality relations:
\begin{theorem}\cite[Theorem 7.3.1 and the subsequent remark]{G-FoundationsTimeFreqAnal}
		\label{l:WZBiOrthRel}
		Let $\psi, \gamma \in  \mathcal{S}$ and let $a,b > 0$. Then,  $(\psi,\gamma)$  is a pair of dual windows on $a\Z \times b\Z$ if and only if
		$$
		\frac{1}{ab}(  T_{\frac{k}{b}}M_{\frac{n}{a}}\psi, T_{\frac{k'}{b}} M_{\frac{n'}{a}} \gamma)_{L^{2}} = \delta_{k,k'} \delta_{n,n'}, \qquad \forall (k, n), (k,' n') \in \Z^2,
$$
or, equivalently,
\begin{equation}
\frac{1}{ab}C^{\frac{1}{b}, \frac{1}{a}}_\psi \circ D^{\frac{1}{b},\frac{1}{a}}_\gamma =  \operatorname{id}_{\ell^2\left( \Z^2 \right)}.
		\label{WR}
\end{equation}
\end{theorem}		

Let $\omega$ and $\eta$ be two weight functions.  The space  $\mathcal{S}^{[\omega]}_{[\eta]}$ is called \emph{Gabor accessible} if there exist $\psi, \gamma \in \mathcal{S}^{[\omega]}_{[\eta]}$ and $a, b > 0$ such that $(\psi,\gamma)$  is a pair of dual windows on $a\Z \times b\Z$. We now use a fundamental result of Janssen \cite{Janssen} (see also \cite{B-J-GaborUnimodWindDecay})  to give a  growth condition on  $\omega$ and $\eta$ which ensures that  $\mathcal{S}^{[\omega]}_{[\eta]}$ is Gabor accessible.

\begin{proposition}\label{GA-1} Let $\omega$ and $\eta$ be two weight functions. The space $\mathcal{S}^{[\omega]}_{[\eta]}$ is Gabor accessible if $\omega(t) = o(t^2)$ and $\eta(t) = o(t^2)$ ($\omega(t) = O(t^2)$ and $\eta(t) = O(t^2)$).
\end{proposition}
\begin{proof}
Since $\mathcal{S}^{1/2}_{1/2} \subseteq \mathcal{S}^{[\omega]}_{[\eta]}$, it suffices to show that $\mathcal{S}^{1/2}_{1/2}$ is Gabor accessible. Let $\psi(x) = e^{-\pi x^2}$, $x \in \R$, be the Gaussian and fix $a,b > 0$ with $ab <1$. Then,  $\psi \in  \mathcal{S}^{1/2}_{1/2}$ and, by \cite[Proposition B and its proof]{Janssen},  there exists $\gamma \in  \mathcal{S}^{1/2}_{1/2}$ such that $(\psi, \gamma)$ is a pair of dual windows on $a\Z \times b\Z$.  
\end{proof}

\begin{remark}
Let $\omega$ and $\eta$ be  two weight functions and suppose that  $\mathcal{S}^{[\omega]}_{[\eta]}$  is non-trivial. It seems an open problem whether  $\mathcal{S}^{[\omega]}_{[\eta]}$ is Gabor accessible. In particular, it seems not to be known whether $\mathcal{S}^\mu_\tau$, $\mu + \tau \geq 1$,  is always Gabor accessible.
\end{remark}
Let $\omega$ and $\eta$ be two weight functions. Finally, we discuss the mapping properties of the analysis and synthesis operators on $\mathcal{S}^{[\omega]}_{[\eta]}$. To this end, we define
$$
\lambda^{(\omega)}_{(\eta)} = \Lambda_\infty(\Z^2;  (\eta(|k|)+ \omega(|n|))_{(k,n) \in \Z^2}) \qquad \mbox{and} \qquad \lambda^{\{\omega\}}_{\{\eta\}} = \Lambda'_0(\Z^2;  (\eta(|k|)+ \omega(|n|))_{(k,n) \in \Z^2}). 
$$
Note that 
$$
 \lambda^{\{\omega\}}_{\{\eta\}} =   \varinjlim_{r \to 0^+} \Lambda^{(\eta(|k|)+ \omega(|n|))_{(k,n) \in \Z^2}}_r(\Z^2).
 $$
 \begin{proposition}\label{mpGabor}
Let $\omega$ and $\eta$ be two weight functions, let $\psi \in \mathcal{S}^{[\omega]}_{[\eta]}$,  and  let $ a,b > 0$.  The mappings
$$
C_\psi = C^{a,b}_\psi: \mathcal{S}^{[\omega]}_{[\eta]} \rightarrow \lambda^{[\omega]}_{[\eta]} \qquad \mbox{and} \qquad 
D_\psi = D^{a,b}_\psi: \lambda^{[\omega]}_{[\eta]} \to \mathcal{S}^{[\omega]}_{[\eta]} 
$$
are continuous.
 \end{proposition}
 \begin{proof}
The continuity of $C^{a,b}_\psi: \mathcal{S}^{[\omega]}_{[\eta]} \rightarrow \lambda^{[\omega]}_{[\eta]}$ follows directly from \cite[Theorem 2.1 and Proposition 2.2(a)]{D-N-V-CharNuclBBSp}.
We now consider $D^{a,b}_\psi$.
By condition $(\alpha)$, it holds that 
 $$
  \lambda^{(\omega)}_{(\eta)} =   \varprojlim_{r \to \infty} \Lambda^{(\eta(a|k|)+ \omega(b|n|))_{(k,n) \in \Z^2}}_r (\Z^2)  \qquad \mbox{and} \qquad 
  \lambda^{\{\omega\}}_{\{\eta\}} =   \varinjlim_{r \to 0^+} \Lambda^{(\eta(a|k|)+ \omega(b|n|))_{(k,n) \in \Z^2}}_r (\Z^2).
$$
Conditions $(\alpha)$ and $(\gamma)$ ($\{\gamma\}$) imply that for some $\rho >0$ (for all $\rho > 0$)
	\[ \sum_{n \in \Z} e^{- \rho \omega(b |n|)} =: C_{\rho,\omega} < \infty \quad \text{and} \quad \sum_{k \in \Z} e^{- \rho \eta(a |k|)} =: C_{\rho,\eta} < \infty . \]
Let $K, C \geq 1$ be such that \eqref{eq:AlphaEquiv} holds. \\
\emph{Beurling case:} It suffices to show that, for all $h >0$, the mapping
\begin{equation}
\label{Dsteps}
 D^{a, b}_{\psi} : \Lambda_{Kh + \rho}^{(\eta(a|k|)+ \omega(b|n|))_{(k,n) \in \Z^2}}(\Z^2) \to \mathcal{S}^{\omega,h}_{\eta, h} 
 \end{equation}
is continuous.  Let  $c \in \Lambda_{Kh + \rho}^{(\eta(a|k|)+ \omega(b|n|))_{(k,n) \in \Z^2}}(\Z^2)$ be arbitrary, and write $\|c\|$ for its norm. 
For all $x \in \R$ it holds that
	\begin{align*}
		|D^{a, b}_{\psi}(c)(x)| 
		&\leq \sum_{(k, n) \in \Z^2} |c_{k, n}| |\psi(x - ak)| \\
		&\leq \|c\| \sum_{(k, n) \in \Z^2} e^{-\rho (\eta(a|k|) + \omega(b|n|))} |\psi(x - ak)| e^{-Kh \eta(a |k|)} \\
		&\leq C C_{\rho,\eta} C_{\rho,\omega} \|\psi\|_{\eta, K h} \|c\| e^{-h\eta(|x|)} . 
	\end{align*}
Since
	\[ \mathcal{F}(D^{a, b}_\psi(c)) = \sum_{(k, n) \in \Z^2} c_{k, n}  M_{-ak}T_{bn} \widehat{\psi} , \]
a similar calculation yields that, for all $\xi \in \R$,
	\[ |\mathcal{F}(D^{a, b}_{\psi}(c))(\xi)| \leq C C_{\rho,\eta} C_{\rho,\omega} \|\widehat{\psi}\|_{\omega, K h} \|c\|  e^{-h\omega(|\xi|)}. \]
This implies the continuity of the mapping \eqref{Dsteps}.	 \\
\emph{Roumieu case:} Choose $h_0 >0$  such that  $\psi \in \mathcal{S}^{\omega, K h_0}_{\eta, K h_0}$. It suffices to show that, for all $0 < h < h_0$ and  $\rho >0$, the mapping \eqref{Dsteps} is continuous. This can be done as in the Beurling case.
 \end{proof}
 
\subsection{Wilson bases}  Let $\psi \in L^2(\R)$. We define
  \begin{align*}
    \psi_{k,0} &= T_{k} \psi, \qquad k \in \Z, \\
    \psi_{k,n} &= \frac{1}{\sqrt{2}} T_{\frac{k}{2}}(M_n+ (-1)^{k+n} M_{-n})\psi, \qquad (k,n)\in\Z\times\N_{>0},
  \end{align*}
and set $\mathcal{W}(\psi) = \{   \psi_{k,n} \mid  (k,n)\in\Z\times\N \}$. We call  $\mathcal{W}(\psi)$ a \emph{Wilson basis} if it is an orthonormal basis in $L^2(\R)$. In such a case, the 
\emph{Wilson analysis operator}
$$
\tilde{C}_\psi:  L^2(\R) \rightarrow \ell^2(\Z \times \N), ~  f \mapsto ((f,\psi_{k,n})_{L^2})_{(k,n) \in \Z \times \N},
$$
and the \emph{Wilson synthesis operator}
$$
\tilde{D}_\psi: \ell^2(\Z \times \N) \rightarrow L^2(\R), ~  (c_{k,n})_{ (k,n) \in \Z \times \N} \mapsto \sum_{ (k,n) \in \Z \times \N} c_{k,n}\psi_{k,n}
$$
are continuous, and
\begin{equation}
\tilde{D}_\gamma \circ \tilde{C}_\psi = \operatorname{id}_{L^2(\R)} \qquad \mbox{and} \qquad \tilde{C}_\gamma \circ \tilde{D}_\psi = \operatorname{id}_{ \ell^2(\Z \times \N)}. 
		\label{comp1}
\end{equation}

Let $\omega$ and $\eta$ be two weight functions. We now discuss when $\mathcal{S}^{[\omega]}_{[\eta]}$ contains an element $\psi$ such that  $\mathcal{W}(\psi)$ is a Wilson basis.   
There is the following celebrated result of Daubechies, Jaffard, and Journ\'e \cite{D-J-J-WilsonBasisExpDecay}:

\begin{proposition}[{\cite[Theorem 4.1]{D-J-J-WilsonBasisExpDecay}}]
 \label{WA-1} There exists $\psi \in \mathcal{S}^{1}_{1}$ such that  $\mathcal{W}(\psi)$ is a Wilson basis.
\end{proposition}

\begin{proposition}\label{WA-2} Let $\omega$ be a non-quasianalytic weight function. There exists $\psi \in \mathcal{D}^{[\omega]}$ such that $\mathcal{W}(\psi)$ is a Wilson basis.
\end{proposition}
\begin{proof}
This can be shown by using the same argument as in the proof of  \cite[Corollary 8.5.5(b)]{G-FoundationsTimeFreqAnal}.
\end{proof}

\begin{remark}\label{remarkWilson}
$(i)$ Let $\psi \in L^2(\R)$ be such that $\mathcal{W}(\psi)$ is a Wilson basis. Then, $\mathcal{G}(\psi,1/2,1)$ is a (tight) Gabor frame \cite[Corollary 8.5.4]{G-FoundationsTimeFreqAnal}\footnote{This part of  \cite[Corollary 8.5.4]{G-FoundationsTimeFreqAnal} holds without the assumption $\psi(t) = \overline{\psi}(-t)$.}. Hence, given two weight functions $\omega$ and $\eta$, the space $\mathcal{S}^{[\omega]}_{[\eta]}$ is Gabor accessible if it contains an element $\psi$ such that  $\mathcal{W}(\psi)$ is a Wilson basis. \\
$(ii)$ Let $\omega$ and $\eta$ be two weight functions and suppose that  $\mathcal{S}^{[\omega]}_{[\eta]}$  is non-trivial. It seems an open problem whether  $\mathcal{S}^{[\omega]}_{[\eta]}$ contains an element $\psi$ such that  $\mathcal{W}(\psi)$ is a Wilson basis. In particular, this question seems open for the spaces $\mathcal{S}^\mu_\tau$, $\mu + \tau \geq 1$. See \cite[Section 3]{B-J-GaborUnimodWindDecay} for some results related to this problem.
\end{remark}
Let $\omega$ and $\eta$ be two weight functions.  We now study the mapping properties of the Wilson analysis and synthesis operators on  $\mathcal{S}^{[\omega]}_{[\eta]}$. Set
$$
\tilde{\lambda}^{(\omega)}_{(\eta)} = \Lambda_\infty( \Z \times \N;  (\eta(|k|)+ \omega(n))_{(k,n) \in  \Z \times \N})$$
and
$$
\tilde{\lambda}^{\{\omega\}}_{\{\eta\}} = \Lambda'_0( \Z \times \N;  (\eta(|k|)+ \omega(n))_{(k,n) \in  \Z \times \N}). 
$$
Note that 
$$
 \tilde{\lambda}^{\{\omega\}}_{\{\eta\}} =   \varinjlim_{r \to 0^+} \Lambda_r^{(\eta(|k|)+ \omega(n))_{(k,n) \in  \Z \times \N}}( \Z \times \N).
 $$
 \begin{proposition}\label{mpWilson}
Let $\omega$ and $\eta$ be two weight functions and let $\psi \in \mathcal{S}^{[\omega]}_{[\eta]}$.  The mappings
$$
\tilde{C}_\psi : \mathcal{S}^{[\omega]}_{[\eta]} \rightarrow \tilde{\lambda}^{[\omega]}_{[\eta]} \qquad \mbox{and} \qquad 
\tilde{D}_\psi : \tilde{\lambda}^{[\omega]}_{[\eta]} \to \mathcal{S}^{[\omega]}_{[\eta]} 
$$
are continuous.
 \end{proposition}
 \begin{proof}
  For  $c = (c_{k,n})_{(k,n)\in \Z^2} \in \C^{\Z^2}$  we define  $\Phi_1(c) \in  \C^{\Z \times \N}$ via
\begin{equation}
\label{def-v}
    \Phi_1(c)_{k,n} =
    \begin{cases}
      c_{2k,0}, & n=0,\\
      \frac{1}{\sqrt{2}} (c_{k,n} + (-1)^{k+n}c_{k,-n}), & n>0.
    \end{cases}
\end{equation}
Then, $\widetilde{C}_{\psi}  = \Phi_1 \circ  C_{\psi}^{\frac{1}{2},1}$.
For $c = (c_{k,n})_{(k,n)\in \Z \times \N} \in \C^{\Z \times \N}$ we define  $\Phi_2(c) \in  \C^{\Z^2}$ via
\begin{equation}
\label{def-w}
    \Phi_2(c)_{k,n} =
    \begin{cases}
      0, & n=0\; \text{and}\; k\;\text{odd}, \\
      c_{\frac{k}{2},0}, & n=0\; \text{and}\; k\;\text{even}, \\
      \frac{1}{\sqrt{2}} c_{k,n} & n>0, \\
      \frac{(-1)^{k+n}}{\sqrt{2}} c_{k,-n} & n<0 .
    \end{cases}
\end{equation}
Then $\widetilde{D}_{\psi}  = D_{\psi, \frac{1}{2},1} \circ \Phi_2$. Since, by condition $(\alpha)$, the mappings
\[
  \Phi_1\colon \lambda^{[\omega]}_{[\eta]}  \rightarrow \tilde{\lambda}^{[\omega]}_{[\eta]}  \qquad \mbox{and} \qquad
  \Phi_2\colon   \tilde{\lambda}^{[\omega]}_{[\eta]} \rightarrow  \lambda^{[\omega]}_{[\eta]}
\]
are continuous, the result follows from Proposition \ref{mpGabor}.
\end{proof}

\section{Sequence space representations of Beurling-Bj\"orck spaces}
\label{sec:SeqSpRep}
In this section, we present our main results: non-constructive and constructive sequence space representations of the spaces $\mathcal{S}^{[\omega]}_{[\eta]}$.

\subsection{Non-constructive sequence space representations of $\mathcal{S}^{[\omega]}_{[\eta]}$}
We are ready to prove the first main result of this article.
\begin{theorem}  \label{seq-space-repr} Let $\omega$ and $\eta$ be  two weight functions and suppose that $\mathcal{S}_{[\eta]}^{[\omega]}$ is Gabor accessible. Then,
\begin{equation}
\label{SSR}
\mathcal{S}_{(\eta)}^{(\omega)} \cong \Lambda_\infty( \alpha(\omega), \alpha(\eta)) \cong \Lambda_\infty (\alpha(\omega,\eta))
\end{equation}
and
\begin{equation}
\label{SSR1}
\mathcal{S}_{\{\eta\}}^{\{\omega\}} \cong \Lambda'_0( \alpha(\omega), \alpha(\eta)) \cong \Lambda'_0 ( \alpha(\omega,\eta)).
\end{equation}
\end{theorem}
\begin{proof}
The second isomorphism in both \eqref{SSR} and \eqref{SSR1} follows from \eqref{basic} and Lemma \ref{explicit}. We now show the first ones. Note that, by condition $(\alpha)$,  $\Lambda_\infty( \alpha(\omega), \alpha(\eta)) \cong   \lambda^{(\omega)}_{(\eta)}$ and $\Lambda'_0( \alpha(\omega), \alpha(\eta)) \cong   \lambda^{\{\omega\}}_{\{\eta\}}$. Hence, by Proposition \ref{iso-sequence}, it suffices to show that $\mathcal{S}_{[\eta]}^{[\omega]}$ is isomorphic to a complemented subspace of  $\lambda^{[\omega]}_{[\eta]}$ and that $\lambda^{[\omega]}_{[\eta]}$ is isomorphic to a complemented subspace of  $\mathcal{S}_{[\eta]}^{[\omega]}$. As  $\mathcal{S}_{[\eta]}^{[\omega]}$ is Gabor accessible, there are $\psi, \gamma \in \mathcal{S}_{[\eta]}^{[\omega]}$ and $a,b >0$ such that $(\psi,\gamma)$ is a pair of dual windows on $a \Z \times b \Z$. Proposition \ref{mpGabor} implies that the mappings
$$
 C^{a,b}_\psi: \mathcal{S}^{[\omega]}_{[\eta]} \rightarrow \lambda^{[\omega]}_{[\eta]} \qquad \mbox{and} \qquad 
D^{a,b}_\gamma: \lambda^{[\omega]}_{[\eta]} \to \mathcal{S}^{[\omega]}_{[\eta]} 
$$
are continuous, and by \eqref{comp11}, it holds that $D^{a,b}_\gamma \circ C^{a,b}_\psi = \operatorname{id}_{\mathcal{S}^{[\omega]}_{[\eta]}}$. This shows that $\mathcal{S}_{[\eta]}^{[\omega]}$ is isomorphic to a complemented subspace of  $\lambda^{[\omega]}_{[\eta]}$. Another application of Proposition \ref{mpGabor} gives that the mappings
$$
 C^{\frac{1}{b}, \frac{1}{a}}_\psi: \mathcal{S}^{[\omega]}_{[\eta]} \rightarrow \lambda^{[\omega]}_{[\eta]} \qquad \mbox{and} \qquad 
D^{\frac{1}{b}, \frac{1}{a}}_\gamma: \lambda^{[\omega]}_{[\eta]} \to \mathcal{S}^{[\omega]}_{[\eta]} 
$$
are continuous, and by \eqref{WR}, it holds that $(ab)^{-1}C^{\frac{1}{b}, \frac{1}{a}}_\psi \circ D^{\frac{1}{b},\frac{1}{a}}_\gamma =  \operatorname{id}_{ \lambda^{[\omega]}_{[\eta]} }$. This shows that $\lambda^{[\omega]}_{[\eta]}$ is isomorphic to a complemented subspace of  $\mathcal{S}_{[\eta]}^{[\omega]}$.
\end{proof}

\begin{corollary} \label{main-text} Let $\omega$ and $\eta$ be two weight functions such that  $\omega(t) = o(t^2)$ and $\eta(t) = o(t^2)$ ($\omega(t) = O(t^2)$ and $\eta(t) = O(t^2)$).
Then, the sequence space representations \eqref{SSR} and \eqref{SSR1} hold.
\end{corollary}
\begin{proof}
This follows from Proposition \ref{GA-1} and Theorem \ref{seq-space-repr}.
\end{proof}
\begin{example}\label{example-iso}  Consider  $\omega_{\mu,u}(t) = t^{\frac{1}{\mu}}\log(1+t)^u$ for $\mu >0$ and $u \in \R$. Note that $\omega_{\mu,u}(t) = o(t^2)$ if and only if $\mu >2$ and $ u \in \R$ or $\mu = 2$ and $u < 0$, while $\omega_{\mu,u}(t) = O(t^2)$ if and only if $\mu >2$ and $u \in \R$ or  $\mu = 2$ and $u \leq 0$. Let $\mu,\tau > 0$ and $u, v \in \R$. If  $\omega_{\mu,u}(t) = o(t^2)$ and  $\omega_{\tau,v}(t) = o(t^2)$, then
$$
\mathcal{S}_{(\omega_{\tau,v})}^{(\omega_{\mu,u})} \cong \Lambda_\infty (  \alpha(\omega_{\mu+\tau, \frac{u\mu + v\tau}{\mu + \tau}})).
$$
 If  $\omega_{\mu,u}(t) = O(t^2)$ and  $\omega_{\tau,v}(t) = O(t^2)$, then
$$
\mathcal{S}_{\{\omega_{\tau,v}\}}^{\{\omega_{\mu,u}\}} \cong \Lambda'_0 (  \alpha(\omega_{\mu+\tau, \frac{u\mu + v\tau}{\mu + \tau}})).
$$
This follows from Example \ref{exampleGStwist} and Corollary \ref{main-text}. In particular, by Example \ref{CGS}, we find, for $\mu, \tau > 1/2$,
$$
\Sigma^{\mu}_{\tau} = \mathcal{S}^{(\omega_{\mu,0})}_{(\omega_{\tau,0})} \cong  \Lambda_\infty (\alpha(\omega_{\mu+\tau,0}))
$$
and, for $\mu, \tau \geq 1/2$,
$$
\mathcal{S}^{\mu}_{\tau} = \mathcal{S}^{\{\omega_{\mu,0}\}}_{\{\omega_{\tau,0}\}} \cong  \Lambda'_0 (\alpha(\omega_{\mu+\tau,0})) . 
$$
This shows Theorem \ref{thm-intro} from the introduction. 
\end{example}

\subsection{Constructive sequence space representations of $\mathcal{S}^{[\omega]}_{[\eta]}$}

For $c = (c_{k,n})_{(k,n)\in \Z \times \N} \in \C^{\Z \times \N}$ we define  $\Phi(c) \in  \C^{\N^2}$ via
\begin{equation}
\label{def-w}
    \Phi(c)_{k,n} =
    \begin{cases}
    c_{-\frac{k}{2},n}, & k\;\text{even}, \\
      c_{\frac{k+1}{2},n}, & k\;\text{odd}.
    \end{cases}
\end{equation}
Let $\omega$ and $\eta$ be two weight functions. Due to condition $(\alpha)$, the mappings 
\begin{equation}
\label{double}
  \Phi\colon  \tilde{\lambda}^{(\omega)}_{(\eta)}  \to \Lambda_\infty( \alpha(\omega), \alpha(\eta)) \qquad \mbox{and} \qquad   \Phi\colon  \tilde{\lambda}^{\{\omega\}}_{\{\eta\}}  \to \Lambda'_0( \alpha(\omega), \alpha(\eta))
\end{equation}
are topological isomorphisms.
We now give the second main result of this article.
\begin{theorem}\label{ssrWilson}
Let $\omega$ and $\eta$ be  two weight functions and let $ \psi \in \mathcal{S}_{[\eta]}^{[\omega]}$ be such that $\mathcal{W}(\psi)$ is a Wilson basis. The mappings
\begin{equation}
\label{Wilsondoublemap}
  \Phi \circ \tilde{C}_\psi \colon  \mathcal{S}_{(\eta)}^{(\omega)}  \to \Lambda_\infty( \alpha(\omega), \alpha(\eta)) \qquad \mbox{and} \qquad    \Phi \circ \tilde{C}_\psi\colon  \mathcal{S}_{\{\eta\}}^{\{\omega\}}  \to \Lambda'_0( \alpha(\omega), \alpha(\eta))
\end{equation}
are topological isomorphisms, and their inverses are given by
\[
   \tilde{D}_\psi  \circ  \Phi^{-1} \colon \Lambda_\infty( \alpha(\omega), \alpha(\eta)) \to   \mathcal{S}_{(\eta)}^{(\omega)}  \qquad \mbox{and} \qquad     \tilde{D}_\psi  \circ  \Phi^{-1} \colon \Lambda'_0( \alpha(\omega), \alpha(\eta))
\to   \mathcal{S}_{\{\eta\}}^{\{\omega\}}.    \] 
\end{theorem}
\begin{proof}
This follows from the equalities \eqref{comp1}, Proposition \ref{mpWilson}, and the fact that the mappings in \eqref{double} are topological isomorphisms.
\end{proof} 
As a corollary, we obtain `commuting' sequence space representations.
\begin{corollary}\label{corWilson}
There exists $\psi \in \mathcal{S}_{1}^{1}$ such that, for all weight functions  $\omega$ and $\eta$ with $\omega(t) = o(t)$  and $\eta(t) = o(t)$, 
  $$
   \Phi \circ \tilde{C}_\psi\colon   \mathcal{S}_{(\eta)}^{(\omega)}  \to \Lambda_\infty( \alpha(\omega), \alpha(\eta)) 
$$
and, for all weight functions  $\omega$ and $\eta$ with $\omega(t) = O(t)$  and $\eta(t) = O(t)$, 
$$ \Phi \circ \tilde{C}_\psi\colon  \mathcal{S}_{\{\eta\}}^{\{\omega\}}  \to \Lambda'_0( \alpha(\omega), \alpha(\eta))
$$
are topological isomorphisms.
\end{corollary}
\begin{proof}
This follows from Proposition \ref{WA-1} and Theorem \ref{ssrWilson}.
\end{proof}
Note that,  in view of Example \ref{CGS}, Theorem \ref{intro-2} from the introduction is a particular instance of Corollary \ref{corWilson}. Finally, we discuss the non-quasianalytic case.
\begin{corollary}\label{ssrWilson-nqa}
Let $\omega$ be a non-quasianalytic weight function.
There exists $\psi \in \mathcal{D}^{[\omega]}$ such that, for all weight functions $\eta$, the mappings in \eqref{Wilsondoublemap} are topological isomorphisms. 
\end{corollary}
\begin{proof}
This follows from Proposition \ref{WA-2} and Theorem \ref{ssrWilson}.
\end{proof} 

\begin{remark}
Let $\omega$ and $\eta$ be two weight functions. The Fourier transform is  a topological isomorphism from $\mathcal{S}_{[\eta]}^{[\omega]}$ onto $\mathcal{S}^{[\eta]}_{[\omega]}$. Hence, if $\eta$ is non-quasianalytic, we may combine Corollary \ref{ssrWilson-nqa} with the Fourier transform to obtain sequence space representations of the spaces  $\mathcal{S}_{[\eta]}^{[\omega]}$. 
\end{remark}

\section{Isomorphic classification of Beurling-Bj\"orck spaces}
\label{sec:IsoClass}

In this section, we use Theorem \ref{seq-space-repr} to study the isomorphic classification of Beurling-Bj\"orck spaces. 
\begin{theorem} \label{iso-BB} 
Let $\omega_i$ and $\eta_i$, $i = 1,2$, be four weight functions such that $\mathcal{S}_{[\eta_1]}^{[\omega_1]}$ and $\mathcal{S}_{[\eta_2]}^{[\omega_2]}$  are Gabor accessible. Then,
$\mathcal{S}_{[\eta_1]}^{[\omega_1]} \cong \mathcal{S}_{[\eta_2]}^{[\omega_2]}$
if and only if there are $L >0$ and $s_0 \geq 0$ such that
\begin{equation}
\label{eq:IsomCond}
\omega^{-1}_2(s/L) \eta^{-1}_2(s/L) \leq \omega^{-1}_1(s) \eta^{-1}_1(s) \leq \omega^{-1}_2(Ls) \eta^{-1}_2(Ls), \qquad s \geq s_0.
\end{equation}
\end{theorem}
\begin{proof}
By Theorem \ref{seq-space-repr}, we have that $\mathcal{S}^{(\omega_i)}_{(\eta_i)} \cong \Lambda_{\infty}(\alpha(\omega_i, \eta_i))$ ($\mathcal{S}^{\{\omega_i\}}_{\{\eta_i\}} \cong \Lambda_0^\prime(\alpha(\omega_i, \eta_i))$) for $i = 1, 2$.
Therefore, \cite[Proposition 29.1]{M-V-IntroFuncAnal} yields that $\mathcal{S}_{[\eta_1]}^{[\omega_1]} \cong \mathcal{S}_{[\eta_2]}^{[\omega_2]}$ if and only if $\alpha(\omega_1, \eta_1) \asymp \alpha(\omega_2, \eta_2)$.
The latter is equivalent to \eqref{eq:IsomCond}.
\end{proof}
We obtain the following somewhat surprising result for weight functions $\omega$ satisfying the slow growth condition $\omega(t^2) = O(\omega(t))$.
\begin{corollary}\label{collapse}
 Let $\omega$ be a weight function satisfying $\omega(t^2) = O(\omega(t))$. For all weight functions $\eta$ with $\omega = O(\eta)$ it holds that
 $\mathcal{S}_{[\eta]}^{[\omega]} \cong \mathcal{S}^{[\omega]}_{[\omega]}$.
\end{corollary}
\begin{proof}
The condition  $\omega(t^2) = O(\omega(t))$ implies that $\omega$ is non-quasianalytic. Hence $\mathcal{S}_{[\kappa]}^{[\omega]}$ is Gabor accessible for all weight functions $\kappa$ (see Proposition \ref{WA-2} and  Remark \ref{remarkWilson}$(i)$).
Let $L \geq 1$ and $t_0 \geq 0$ be such that $\omega(t^2) \leq L\omega(t)$ and $\omega(t) \leq L \eta(t)$ for all $t \geq t_0$.
Then, there is $s_0 \geq 0$ such that $\omega^{-1}(s)^2 \leq \omega^{-1}(Ls)$ and $\eta^{-1}(s / L) \leq \omega^{-1}(s)$ for  all $s \geq s_0$. By choosing $s_0$ large enough, we may also suppose that $\eta^{-1}(Ls) \geq 1$ for all $s \geq s_0$. Hence, for $s \geq s_0$,
	\[ \omega^{-1}(s/L) \eta^{-1}(s/L) \leq \omega^{-1}(s) \eta^{-1}(s/L) \leq \omega^{-1}(s)^2 \leq \omega^{-1}(Ls) \leq \omega^{-1}(Ls) \eta^{-1}(Ls) ,   \]
and the result follows from Theorem \ref{iso-BB}.
\end{proof}
\begin{example}
The weight functions $\omega_u(t) = \log(1+t)^u$, $u \geq 1$, satisfy $\omega_u(t^2) = O(\omega_u(t))$.
\end{example}
We now show that the phenomenon in Corollary \ref{collapse}
cannot happen if the weight function $\omega$ grows sufficiently fast.

\begin{corollary}
 Let $\omega$ be a weight function such that $\liminf_{t \to \infty} \omega(Ht)/\omega(t)  >  1$ for some $H > 0$. Let $\eta_1, \eta_2$ be two weight functions such that  $\mathcal{S}_{[\eta_1]}^{[\omega]}$ and $\mathcal{S}_{[\eta_2]}^{[\omega]}$ are Gabor accessible. Then,
$\mathcal{S}_{[\eta_1]}^{[\omega]} \cong \mathcal{S}_{[\eta_2]}^{[\omega]}$
 if and only if $\eta_1 \asymp \eta_2$.
\end{corollary}
\begin{proof}
If $\eta_1 \asymp \eta_2$, then $\mathcal{S}_{[\eta_1]}^{[\omega]}$ 	and $ \mathcal{S}_{[\eta_2]}^{[\omega]}$ clearly coincide as locally convex spaces.
Suppose now that $\mathcal{S}_{[\eta_1]}^{[\omega]} \cong \mathcal{S}_{[\eta_2]}^{[\omega]}$.
By Theorem \ref{iso-BB}, there are $L >0$ and $s_0 \geq 0$ such that \eqref{eq:IsomCond} holds with $\omega_1 = \omega_2 = \omega$.
The assumption on $\omega$ implies that there are $K >0$ and $t_0 \geq 0$ such that $L \omega(t) \leq \omega(K t)$ for all $t \geq t_0$.
Then, possibly by raising $s_0$, we find that $\omega^{-1}(L s) \leq K \omega^{-1}(s)$ for all $s \geq s_0$.
Therefore, for $s \geq s_0$,
	\[ \omega^{-1}(s) \eta_1^{-1}(s) \leq \omega^{-1}(Ls) \eta_2^{-1}(Ls) \leq K \omega^{-1}(s) \eta_2^{-1}(Ls) , \]
whence $\eta_1^{-1}(s) \leq K \eta_2^{-1}(Ls)$. This implies that $\eta_2(t) \leq L \eta_1(K t)$ for $t$ large enough.
Since $\eta_1$ satisfies $(\alpha)$, we obtain that $\eta_2 = O(\eta_1)$.
Analogously, one shows that $\eta_1 = O(\eta_2)$.
\end{proof}
\begin{example}
The weight functions $\omega_{\mu,u}(t) = t^{\frac{1}{\mu}}\log(1+t)^u$, $\mu >0$ and $u \in \R$, satisfy $\liminf_{t \to \infty} \omega_{\mu, u}(Ht)/\omega_{\mu, u}(t)  >  1$ for some $H > 0$.
\end{example}
Finally, we characterize when two Beurling-Bj\"orck spaces are equal to each other as sets (see also \cite{D-N-V-InclGSSp}). We include this result here as it seems interesting to compare it with Theorem \ref{iso-BB}, which characterizes when two Beurling-Bj\"orck spaces are topologically isomorphic to each other.
\begin{theorem}\label{inclusion} Let $\omega_i$ and $\eta_i$, $i = 1,2$, be four weight functions and suppose that $\mathcal{S}_{[\eta_1]}^{[\omega_1]}$ and $\mathcal{S}_{[\eta_2]}^{[\omega_2]}$ are non-trivial.
\begin{itemize}
\item[$(i)$] $\mathcal{S}_{[\eta_1]}^{[\omega_1]} \subseteq \mathcal{S}_{[\eta_2]}^{[\omega_2]}$ if  and only if $\omega_2 = O(\omega_1)$ and  $\eta_2 = O(\eta_1)$.
\item[$(ii)$] $\mathcal{S}_{[\eta_1]}^{[\omega_1]} = \mathcal{S}_{[\eta_2]}^{[\omega_2]}$ if and only if $\omega_1 \asymp \omega_2$ and $\eta_1 \asymp \eta_2$.
\end{itemize}
\end{theorem}
\begin{proof}
$(i)$ Sufficiency is clear.
Suppose now that $\mathcal{S}_{[\eta_1]}^{[\omega_1]} \subseteq \mathcal{S}_{[\eta_2]}^{[\omega_2]}$.  Since the Fourier transform is  an isomorphism from $\mathcal{S}_{[\eta_i]}^{[\omega_i]}$ onto $\mathcal{S}^{[\eta_i]}_{[\omega_i]}$, $i = 1,2$, it suffices to show that $\eta_2 = O(\eta_1)$. Let $K,C \geq 1$ be such that \eqref{eq:AlphaEquiv} holds for $\omega = \eta_1$. \\
\emph{Beurling case:} Pick $\psi \in \mathcal{S}_{(\eta_1)}^{(\omega_1)}$ with $\psi(0) = 1$.  By the closed graph theorem, the inclusion $\mathcal{S}_{(\eta_1)}^{(\omega_1)} \subseteq \mathcal{S}_{(\eta_2)}^{(\omega_2)}$ is continuous. Hence, for all $h > 0$ there are $C_1, k > 0$ such that
 $$
\|\varphi\|_{\mathcal{S}^{\eta_2,h}_{\omega_2,h}} \leq C_1\|\varphi\|_{\mathcal{S}^{\eta_1,k}_{\omega_1,k}}, \qquad \forall \varphi \in \mathcal{S}^{(\omega_1)}_{(\eta_1)}.
$$
For every $x \in \R$ it holds that
$$
e^{h\eta_2(x)} \leq \| T_x \psi \|_{\mathcal{S}^{\omega_2,h}_{\eta_2,h}} \leq C_1\|T_x \psi\|_{\mathcal{S}^{\omega_1,k}_{\eta_1,k}} \leq CC_1\|\psi\|_{\mathcal{S}^{\omega_1,Kk}_{\eta_1,Kk}} e^{Kk\eta_1(x)},
$$
which implies  $\eta_2 = O(\eta_1)$. \\
\emph{Roumieu case:} Pick $\psi \in \mathcal{S}_{\{\eta_1\}}^{\{\omega_1\}}$ with $\psi(0) = 1$. Let $k >0$   be such that  $\psi \in \mathcal{S}^{\omega_1, K k}_{\eta_1, K k}$. By  the closed graph theorem of De Wilde, the inclusion $\mathcal{S}_{\{\eta_1\}}^{\{\omega_1\}}\subseteq  \mathcal{S}_{\{\eta_2\}}^{\{\omega_2\}}$ is continuous. Grothendieck's factorization theorem implies  that there are $C_1,h >0$ such that $\mathcal{S}_{\eta_1,k}^{\omega_1,k}\subseteq  \mathcal{S}_{\eta_2,h}^{\omega_2,h}$ and 
$$
\|\varphi\|_{\mathcal{S}^{\eta_2,h}_{\omega_2,h}} \leq C_1\|\varphi\|_{\mathcal{S}^{\eta_1,k}_{\omega_1,k}}, \qquad \forall \varphi \in \mathcal{S}^{\omega_1,k}_{\eta_1,k}.
$$
The result can now be shown in the same way as in the Beurling case. \\
$(ii)$ This follows from part $(i)$.
\end{proof}

\begin{example}
Consider  $\omega_{\mu,u}(t) = t^{\frac{1}{\mu}}\log(1+t)^u$ for $\mu >0$ and $u \in \R$.  Let $\mu_i,\tau_i > 0$ and $u_i,v_i \in \R$, $i =1,2$, be such that $\omega_{\mu_i,u_i}(t) = o(t^2)$ and  $\omega_{\tau_i,v_i}(t) = o(t^2)$ ($\omega_{\mu_i,u_i}(t) = O(t^2)$ and  $\omega_{\tau_i,v_i}(t) = O(t^2)$). Theorem  \ref{iso-BB} (see also Example \ref{example-iso}) yields that
$$
\mathcal{S}_{[\omega_{\tau_1,v_1}]}^{[\omega_{\mu_1,u_1}]} \cong \mathcal{S}_{[\omega_{\tau_2,v_2}]}^{[\omega_{\mu_2,u_2}]}
$$
if and only if $\mu_1 + \tau_1 = \mu_2 + \tau_2$ and $u_1\mu_1 + v_1\tau_1 = u_2\mu_2 + v_2\tau_2$. On the other hand, by Theorem \ref{inclusion}, it holds that 
$$
\mathcal{S}_{[\omega_{\tau_1,v_1}]}^{[\omega_{\mu_1,u_1}]} = \mathcal{S}_{[\omega_{\tau_2,v_2}]}^{[\omega_{\mu_2,u_2}]}
$$
if and only if $\mu_1 = \mu_2$, $\tau_1 = \tau_2$, $u_1 = u_2$, and $v_1 = v_2$.
\end{example}

\end{document}